\documentclass[11pt]{article}
\usepackage{algpseudocode}
\usepackage[ruled]{algorithm}
\usepackage{amsthm}
\usepackage{url}
\usepackage{longtable}
\usepackage{rotating}
\usepackage{graphicx}
\usepackage{color}
\usepackage{latexsym}
\usepackage{amsfonts}
\usepackage{amsmath}
\usepackage{psfrag}
\usepackage{mathtools}
\usepackage{caption}
\usepackage{enumerate}
\usepackage{epstopdf}
\usepackage{pgfplots}
\usepackage{braket}
\usepackage{booktabs}
\usepackage{multirow}
\usepackage{array}
\usepackage[square,numbers]{natbib}
\pgfplotsset{compat=newest}
\newtheorem{proposition}{Proposition}

\newtheorem{theorem}{Theorem}

\newtheorem{lemma}{Lemma}

\newtheorem{definition}{Definition}
\newtheorem{remark}{Remark}

\DeclarePairedDelimiter{\norm}{\lVert}{\rVert} 
\DeclarePairedDelimiter{\abs}{\lvert}{\rvert} 
\DeclareMathOperator{\conv}{conv} 
\newcommand{\R}{\mathbb{R}} 

 \setbox\strutbox=\hbox{\vrule height7pt depth2pt width0pt}

\def\square{{\setbox0=\hbox{X}\hbox to \ht0{\vrule\hss\vbox to \ht0{
  \hrule width \ht0\vfil\hrule width \ht0}\vrule}}}

\makeatletter

\begin{document}
\pagestyle{empty}

\vskip 2cm \begin{center}{\huge Data Filtering for Cluster Analysis by $\ell_0$-Norm Regularization}\end{center}
\par\bigskip
\centerline{\Large A. Cristofari$^*$}
\par\bigskip\bigskip

\centerline{$^*$ Department of Computer, Control, and Management Engineering}
\centerline{Sapienza University of Rome}
\centerline{Via Ariosto, 25, 00185 Rome, Italy}
\par\medskip
\centerline{e-mail (Cristofari): cristofari@dis.uniroma1.it}

\par\bigskip\noindent {\small \centerline{\bf Abstract}}
A data filtering method for cluster analysis is proposed, based on minimizing a least squares function with a weighted $\ell_0$-norm penalty.
To overcome the discontinuity of the objective function, smooth non-convex functions are employed to approximate the $\ell_0$-norm.
The convergence of the global minimum points of the approximating problems towards global minimum points of the original problem is stated.
The proposed method also exploits a suitable technique to choose the penalty parameter.
Numerical results on synthetic and real data sets are finally provided,
showing how some existing clustering methods can take advantages from the proposed filtering strategy.

\pagestyle{plain} \setcounter{page}{1}

\section{Motivation}\label{introduction}
Cluster analysis is a branch of unsupervised learning, arising in many real-world applications and in different fields,
e.g., biology, medicine, marketing, document retrieval, image segmentation and many others.
It deals with grouping objects so that ``alike'' data are in the same clusters and ``unlike'' data are in different clusters.
More formally, given a finite set of vectors \mbox{$X = \{x_1,\dots,x_m\} \subset \R^n$}, we want to divide $X$ into $k$ groups (clusters), according to a defined measure of similarity, where $k$ can be either known or unknown.

Partitioning $X$ into a fixed number of clusters is known to be an NP-hard problem~\cite{garey:1979} and many existing clustering models are formulated as non-convex optimization problems. As a result, algorithms can generally find only approximate solutions. Moreover, there is no objectively ``right'' clustering model and the choice of the most suitable algorithm can strongly depend on the specific data set. So, there is still a great interest in developing new strategies for cluster analysis, also in the field of numerical optimization.

Here, we propose a data filtering method based on combining two different techniques.

The first one is a reformulation of the clustering problem as a penalized regression problem, proposed in~\cite{pelckmans:2005,hocking:2011,lindsten:2011}
and further studied in~\cite{pan:2013,chi:2014,marchetti:2014}.
Assuming that the number of clusters is unknown, this approach is based on introducing for each observation $x_i$ a centroid \mbox{$z_i \in \R^n$}, representing the cluster which $x_i$ belongs to. The problem consists in minimizing the distances between $x_i$ and $z_i$, trying at the same time to group centroids. This is obtained by adding to the objective function a term to penalize each pair $(i,j)$ such that $z_i \neq z_j$. The problem can be formulated as
\begin{equation}\label{general_model}
\min_{z \in \R^{mn}} \sum_{i=1}^m \norm{x_i - z_i}^2 + \lambda\sum_{j=2}^m\sum_{i=1}^{j-1}w_{ij}P(z_i-z_j),
\end{equation}
where we indicate with $z$ the vector $\begin{bmatrix}z_1^T & \dots & z_m^T\end{bmatrix}^T \in \R^{mn}$, $\lambda$ is a nonnegative penalty parameter,
$w_{ij}$ are nonnegative fixed parameters and $P\colon\R^n \rightarrow \R$ is a (symmetric) penalty function such that
\[
P(y)
\begin{cases}
=0, \quad &\text{if } y = 0, \\
> 0, \quad &\text{otherwise}.
\end{cases}
\]

The centroids provided by the solution $z^* = \begin{bmatrix}(z^*_1)^T & \dots & (z^*_m)^T\end{bmatrix}^T$ of~\eqref{general_model} represent the final clusters. Namely, $x_i$ and $x_j$ are in the same cluster if $z^*_i = z^*_j$.

The basic idea behind model~\eqref{general_model} is that a major number of centroids can be grouped simply by increasing the penalty parameter $\lambda$.

Anyway, when a fixed number of clusters is required, choosing a proper value of~$\lambda$ can be a very hard issue.
In fact, by increasing $\lambda$, we can have a larger number of pairs of coinciding centroids in the optimal solution, that is, a larger number of pairs of points that belong to the same cluster. But this does not provide information on the number of clusters we obtain. Consequently, a value of $\lambda$ that produces the desired number of clusters may not even exist.

Here, addressing the case in which a fixed number of clusters is required, we reinterpret model~\eqref{general_model} as a method to map each sample
$x_i$ by a vector $z_i$ that is representative of the local density of the samples in its neighborhood.

The proposed strategy also exploits a suitable technique to choose $\lambda$, based on minimizing a further optimality criterion that considers
the distances within and between clusters.

As regards the penalty function in~\eqref{general_model}, most authors focused on using convex \mbox{$\ell_q$-norms} (e.g., the {$\ell_1$-norm},
or the $\ell_2$-norm), so that problem~\eqref{general_model} is convex.
In order to avoid the bias generated by convex penalties \cite{fan:2001,shen:2012}, some non-convex ones were proposed in~\cite{pan:2013,marchetti:2014}.
On the other hand, the latter have the disadvantage not to make possible to reach the global minimum.

Here, we start from the following observation: since the penalty term in~\eqref{general_model} has only the goal to force some pairs of centroids to coincide, then $P(z_i - z_j)$ should assume a constant value if $z_i \ne z_j$ (i.e., if $x_i$ and $x_j$ are in different clusters), regardless how far $z_i$ and $z_j$
are from each other. Furthermore, the penalty associated with each pair $(i,j)$ should be weighted by taking into account the distance (i.e., the similarity)
between the samples $x_i$ and $x_j$, so that close pairs of points are encouraged to be in the same cluster.

Therefore, weighted $\ell_0$-norm penalties are employed in this paper.
To overcome the non-continuity of the objective function, the $\ell_0$-norm is then approximated with a sequence of smooth non-convex functions that
converges to the $\ell_0$-norm pointwise.
As to be shown, the convergence of the global optimal solutions of the approximating problems towards global optimal solutions of the original problem
can be proved.

The rest of the paper is organized as follows. In Section~\ref{sec:model}, we introduce the \mbox{$\ell_0$-norm} penalty clustering model and its smooth approximation, discussing some theoretical aspects. In Section~\ref{sec:filt_method}, we present the data filtering method.
In Section~\ref{sec:numerical_results}, we show the numerical results.
Finally, in Section~\ref{sec:conclusions}, we draw some conclusions.

From now on, we indicate with $\norm{\cdot}$ the Euclidean norm. Given $v \in \R^n$, we indicate with $(v)_h$ the $h$-th component of $v$, and with $\mathcal{B}(v,\rho)$ the open ball with center $v$ and radius $\rho$. Given a set $S$, we indicate with $\abs S$ its cardinality.

\section{The Model}\label{sec:model}
In this section, we introduce the clustering model with $\ell_0$-norm regularization and its smooth approximation, pointing out the relations between them.
Since this is only the starting point for the proposed data filtering method, we do not address the issues concerning the choice of the penalty parameter, that will be discussed in Section~\ref{sec:filt_method}.

\subsection{The $\ell_0$-Regularized Least Squares Problem}\label{sub:the_clustering_model}
Let $X = \{x_1,\dots,x_m\} \subset \R^n$ be a finite set of vectors and let us consider problem~\eqref{general_model}. As discussed in the previous section, our goal is to employ a penalty function satisfying the following condition for each pair $(i,j)$:
\[
P(z_i-z_j) =
\begin{cases}
0, & \quad \text{ if } z_i = z_j, \\
1, & \quad \text{otherwise},
\end{cases}
\]
that is, $P(z_i-z_j)$ must not depend on the distance between the centroids $z_i$ and $z_j$.
We also want to weigh $P(z_i-z_j)$ by a parameter $w_{ij}$ that takes into account the proximity of the samples $x_i$ and $x_j$.
In particular, $w_{ij}$ should be large if the samples $x_i$ and $x_j$ are near each other, so that close pairs of points are more strongly encouraged to be in the same cluster.

In other words, we want that the penalty value associated with each pair $(i,j)$ depends on the distance between the samples $x_i$ and $x_j$, but not on the distance between the centroids $z_i$ and $z_j$.
This leads to formulate the problem as follows:

\begin{equation}\label{zero-norm_model}
\min_{z \in \R^{mn}} \sum_{i=1}^m \norm{x_i - z_i}^2 + \lambda\sum_{j=2}^m\sum_{i=1}^{j-1}w_{ij}\,s\bigl(\norm{z_i-z_j}\bigr),
\end{equation}
where $s\colon\R\rightarrow\{0,1\}$ is the step function defined as
\begin{equation}\label{step_funct}
s(u) =
\begin{cases}
0, \quad &\text{if } u = 0, \\
1, \quad &\text{otherwise},
\end{cases}
\end{equation}
and $w_{ij}$ are inversely proportional to the distance between $x_i$ and $x_j$.

We observe that the penalty term can be seen as a weighted $\ell_0$-norm of the vector with components $\norm{z_i-z_j}$.
Namely, we seek a solution $z^*$ minimizing $\sum_{i=1}^m \norm{x_i - z_i}^2$, such that the vector $\begin{bmatrix}\norm{z_i - z_j}\end{bmatrix}_{i<j}$ is sufficiently sparse.

\begin{remark}
Problem~\eqref{zero-norm_model} is well defined, in the sense that it attains a minimizer, since the objective function is lower semicontinuous and coercive~\cite{rockafellar:2009}.
\end{remark}

\subsection{The Smooth Approximating Problem}
Minimizing a non-continuous function is hard, then it is reasonable trying to approximate~\eqref{zero-norm_model} with a continuous and smooth problem.

Indicating with $\phi(z)$ the objective function of~\eqref{zero-norm_model}, we seek a smooth function $g(z;\alpha)$, depending on a parameter $\alpha$,
that converges to $\phi(z)$ pointwise. Namely, there must exist a sequence $\bigl\{\alpha^t\bigr\}$ such that
\begin{equation}\label{lim_app}
\lim_{t \to \infty} g\bigl(z;\alpha^t\bigr) = \phi(z), \quad \forall z \in \R^{mn}.
\end{equation}
Roughly speaking, we expect that the minimum points of $g\bigl(z;\alpha^t\bigr)$ are ``similar'' to those of $\phi(z)$ for suitable values of the index $t$.

Many smooth approximations of the $\ell_0$-norm were proposed in the literature. In particular, since the $\ell_0$-norm of a vector is given by the sum of step functions,
in~\cite{mangasarian:1996,bradley:1998} the authors approximated the step function~\eqref{step_funct}
with the following concave parametric function:
\begin{equation}\label{step_funct_approx_mangasarian}
\beta(u;\alpha) = 1 - e^{-\alpha u}, \quad u \ge 0, \quad \alpha > 0.
\end{equation}
This approach can be convenient when minimizing the $\ell_0$-norm of a vector over a polyhedral set admitting a vertex.
Exploiting the concavity of~\eqref{step_funct_approx_mangasarian},
it can be proved that there exists a finite index $\bar{t}$ such that, for every $t\ge\bar{t}$, the optimal solutions of the approximating problem also solve the original problem~\cite{rinaldi:2010}.

In our case, we are not interested in approximating~\eqref{step_funct} with a concave function, because the least squares term would make the approximating problem non-concave anyway.
So, we slightly adapt the above described approach and we approximate
the term $\displaystyle \sum_{j=2}^m\sum_{i=1}^{j-1}s\bigl(\norm{z_i-z_j}\bigr)$ with the following smooth parametric function:
\[
\gamma(z;\alpha) = \sum_{j=2}^m\sum_{i=1}^{j-1}\Bigl(1-e^{-\alpha \norm{z_i-z_j}^2}\Bigr), \quad \alpha > 0.
\]


We finally write the problem approximating~\eqref{zero-norm_model} as
\begin{equation}\label{zero-norm_model_approx}
\min_{z \in \R^{mn}} \sum_{i=1}^m\norm{x_i-z_i}^2 + \lambda\sum_{j=2}^m\sum_{i=1}^{j-1} w_{ij} \Bigl(1-e^{-\alpha \norm{z_i-z_j}^2}\Bigr).
\end{equation}
Indicating with $g(z;\alpha)$ the objective function of~\eqref{zero-norm_model_approx}, it is straightforward to verify
that~\eqref{lim_app} holds for every sequence $\bigl\{\alpha^t\bigr\}$ such that $\displaystyle \lim_{t \to \infty} \alpha^t = +\infty$.
Then, we expect that the larger $\alpha$ is, the better~\eqref{zero-norm_model_approx} approximates~\eqref{zero-norm_model}.

Finally, let us remark that our approximation does not require slack variables and feasibility constraints.

\subsection{Properties of the Approximating Problem}\label{sub:properties_approx}
In this subsection, we investigate some theoretical properties of problem~\eqref{zero-norm_model_approx}, pointing out the relations between its optimal solutions and those of~\eqref{zero-norm_model}. To this aim, we briefly recall the definition of the projection operator and we state some preliminary lemmas.
\begin{definition}
Let $C \subseteq \R^n$ be a non-empty closed convex set.
Given $x \in \R^n$, we call \textit{projection of $x$ on $C$} the unique solution $p(x)$ of the problem
\[
\min \, \{\norm{x-y} \, \colon \, y\in C\}.
\]
\end{definition}

\begin{lemma}\label{lemma:proj_properties1}
Let $C \subseteq \R^n$ be a non-empty closed convex set.
\begin{itemize}
\item For any $x \in \R^n$, $p(x)$ is the projection of $x$ on $C$ if and only if
    \begin{equation}\label{proj_dist1_a}
    (x-p(x))^T(y-p(x)) \le 0, \quad \forall y \in C.
    \end{equation}
\item For any $x,y \in \R^n$, let $p(x)$ and $p(y)$ be the projections of $x$ and $y$ on $C$, respectively. Then,
    \begin{equation}\label{proj_dist1_b}
    \norm{p(x)-p(y)} \le \norm{x-y}.
    \end{equation}
\end{itemize}
\end{lemma}

\begin{proof}
See \cite{bertsekas:1999}[Proposition 2.1.3].
\end{proof}

\begin{lemma}\label{lemma:proj_properties2}
Let $C \subset \R^n$ be a non-empty closed convex set. Given $x \in C$ and $y \in \R^n \setminus C$, let $p(y)$ be the projection of $y$ on $C$. Then,
\begin{equation}\label{proj_dist2}
\norm{x-(y+\xi(p(y)-y))} < \norm{x-y}, \quad \forall \xi \in (0,1].
\end{equation}
\end{lemma}
\begin{proof}
Let $\tilde y = y+\xi(p(y)-y)$, where $\xi \in (0,1]$. We can write:
\begin{align*}
x - y        & = (x - p(y)) + (p(y) - y), \\
x - \tilde y & = (x - p(y)) + (p(y) - \tilde y) = (x - p(y)) + (1 - \xi)(p(y) - y).
\end{align*}
From the above relations, it follows that
\begin{align*}
\norm{x - y}^2        & = \norm{x-p(y)}^2 + \norm{p(y)-y}^2 + 2(x-p(y))^T(p(y)-y), \\
\norm{x - \tilde y}^2 & = \norm{x-p(y)}^2 + (1 - \xi)^2 \norm{p(y)-y}^2 + 2(1 - \xi)(x-p(y))^T(p(y)-y).
\end{align*}
Consequently,
\[
\norm{x - y}^2 - \norm{x - \tilde y}^2 = \bigl(1 - (1 - \xi)^2\bigr)\norm{p(y)-y}^2 + 2\xi (x-p(y))^T(p(y)-y).
\]
Since $y \notin C$, $\xi \in (0,1]$, and taking into account~\eqref{proj_dist1_a} of Lemma~\ref{lemma:proj_properties1},
we obtain that $\norm{x - y}^2 - \norm{x - \tilde y}^2 > 0$.
\end{proof}

\begin{lemma}\label{lemma:proj_properties3}
Let $C \subseteq \R^n$ be a non-empty closed convex set.
Given $x,y \in \R^n$, let $p(x)$ and $p(y)$ be the projections of $x$ and $y$ on $C$, respectively. Then,
\begin{equation}\label{proj_dist3}
\norm{(x+\xi(p(x)-x))-(y+\xi(p(y)-y))} \le \norm{x-y}, \quad \forall \xi \in [0,1].
\end{equation}
\end{lemma}

\begin{proof}
Let us consider the function $\omega\biggl(\begin{bmatrix}u \\ v\end{bmatrix}\biggr) = \norm{u-v}$,
where $u,v \in \R^n$.

\noindent
Since $\omega$ is convex in $\R^{2n}$, for all $\xi \in [0,1]$ we can write
\[
\begin{split}
& \norm{(x+\xi(p(x)-x))-(y+\xi(p(y)-y))} = \omega\biggl(\begin{bmatrix}x \\ y\end{bmatrix}+\xi \begin{bmatrix}p(x)-x \\ p(y)-y\end{bmatrix}\biggr) \\
& = \omega\biggl((1-\xi)\begin{bmatrix}x \\ y\end{bmatrix}+\xi \begin{bmatrix}p(x) \\ p(y)\end{bmatrix}\biggr)
  \le (1-\xi)\ \omega\biggl(\begin{bmatrix}x \\ y\end{bmatrix}\biggr) + \xi\ \omega \biggl(\begin{bmatrix}p(x) \\ p(y)\end{bmatrix}\biggr) \\
& = (1-\xi)\norm{x-y} + \xi\norm{p(x)-p(y)} \le (1-\xi)\norm{x-y} + \xi\norm{x-y} = \norm{x-y},
\end{split}
\]
where the last inequality follows from~\eqref{proj_dist1_b} of Lemma~\ref{lemma:proj_properties1}.
\end{proof}

Now, we can start analyzing some properties of problem~\eqref{zero-norm_model_approx}. First, it attains optimal solutions, since the objective function is coercive. Moreover, the next proposition claims that all the local optimal solutions of~\eqref{zero-norm_model_approx}
are contained in a compact set, which does not depend on $\lambda$ and $\alpha$.

\begin{proposition}\label{prop:local_opt_sol_zero-norm_model_approx}
Given a finite set of vectors $X = \{x_1,\dots,x_m\} \subset \R^n$, $\alpha>0$, $\lambda \ge 0$, $w_{ij}\ge0, \, j=2,\dots,m, \, i=1,\dots,j-1$, let $z^* = \begin{bmatrix}(z_1^*)^T & \dots & (z_m^*)^T\end{bmatrix}^T$ be a local optimal solution of~\eqref{zero-norm_model_approx}. Then, $z^*_1,\dots,z^*_m$ are in the convex hull of $X$.
\end{proposition}

\begin{proof}
Let $g(z;\alpha)$ be the objective function of problem~\eqref{zero-norm_model_approx} for any parameter \mbox{$\alpha > 0$.} Proceeding by contradiction, we assume that
$z^* = \begin{bmatrix}(z^*_1)^T & \dots & (z^*_m)^T\end{bmatrix}^T$ is a local optimal solution of~\eqref{zero-norm_model_approx} and the following index subset is non-empty:
\[
I = \bigl\{h \in \{1,\dots,m\} \colon z^*_h \notin \conv{(X)}\},
\]
where $\conv(X)$ is the convex hull of $\{x_1,\dots,x_m\}$. We assume without loss of generality that $I = \{1,\dots,\abs{I}\}$.

Any vector $z \in \R^{mn}$ can be written as $z = \begin{bmatrix} z(I)^T &  & z(N)^T \end{bmatrix}^T$,
where
\[
z(I) = \begin{bmatrix}z_1^T & \dots & z_{\abs{I}}^T\end{bmatrix}^T \qquad \text{ and } \qquad z(N) = \begin{bmatrix}z_{\abs{I}+1}^T & \dots & z_m^T\end{bmatrix}^T.
\]
So, in the following we indicate with $z^*(I)$ the vector $\begin{bmatrix}(z^*_1)^T & \dots & (z^*_{\abs{I}})^T\end{bmatrix}^T$, and with $z^*(N)$ the vector $\begin{bmatrix}(z^*_{\abs{I}+1})^T & \dots & (z^*_m)^T\end{bmatrix}^T$.

For each $i = 1,\dots,m$, we compute $p(z^*_i)$ as the projection of $z^*_i$ on $\conv{(X)}$.
Now, we define the vector $\bar d = \begin{bmatrix} (\bar d_1)^T & \dots & (\bar d_m)^T \end{bmatrix}^T \in \R^{mn}$ such that
\[
\bar d_i = p(z^*_i) - z^*_i \in \R^n, \quad i = 1,\dots,m,
\]
and we rewrite $\bar d$ as $\begin{bmatrix}\bar d(I)^T &  & \bar d(N)^T\end{bmatrix}^T$, 
where
\[
\bar d(I) = \begin{bmatrix}\bar d_{1}^{\,T} & \dots & \bar d_{\abs I}^{\,T}\end{bmatrix}^T \qquad \text{ and } \qquad \bar d(N) = \begin{bmatrix}\bar d_{\abs I+1}^{\,T} & \dots & \bar d_m^{\,T}\end{bmatrix}^T.
\]
From the definition of $I$, it follows that $\bar d \ne 0$. In particular, we have
\begin{align*}
\bar d_i \neq 0, \quad & i = 1,\dots,\abs I, \\
\bar d_i = 0, \quad & i = \abs I+1,\dots,m.
\end{align*}

We show that $\bar d$ is a descent direction for $g(z;\alpha)$ at $z^*$, namely, that there exists a scalar $\bar \xi > 0$ such that
\begin{equation}\label{proof:desc2}
g(z^* + \xi \bar d;\alpha) < g(z^*;\alpha), \quad \forall \xi \in (0,\bar \xi].
\end{equation}
To this aim, we rewrite $g(z;\alpha) = g_1\bigl(z(I)\bigr) + g_2\bigl(z(N)\bigr) + g_3(z;\alpha)$,
where
\begin{align*}
& g_1\bigl(z(I)\bigr) = \sum_{j=1}^{\abs{I}} \norm{x_j - z_j}^2, \\
& g_2\bigl(z(N)\bigr) = \sum_{j=\abs{I}+1}^m \norm{x_j - z_j}^2, \\
& g_3(z;\alpha) = \lambda\sum_{j=2}^m\sum_{i=1}^{j-1} w_{ij} \Bigl(1-e^{-\alpha \norm{z_i-z_j}^2}\Bigr).
\end{align*}
We consider $g_1\bigl(z(I)\bigr)$, $g_2\bigl(z(N)\bigr)$ and $g_3(z;\alpha)$ separately.
\begin{itemize}
\item First, we consider $g_1\bigl(z(I)\bigr)$. From Lemma~\ref{lemma:proj_properties2}, for all $j \in I$ we can write
    \[
    \norm{x_j - (z^*_j + \xi \bar d_j)}^2 < \norm{x_j-z^*_j}^2, \quad \forall \xi \in (0,1],
    \]
    and then
    \begin{equation}\label{proof:decr4}
    g_1\bigl(z^*(I) + \xi \bar d(I)\bigr) < g_1\bigl(z^*(I)\bigr), \quad \forall \xi \in (0,1].
    \end{equation}
\item Now, we consider $g_2\bigl(z(I)\bigr)$. Since $\bar d(N) = 0$, we simply have
    \begin{equation}\label{proof:decr5}
    g_2\bigl(z^*(N) + \xi \bar d(N)\bigr) = g_2\bigl(z^*(N)\bigr), \quad \forall \xi > 0.
    \end{equation}
\item Finally, we consider $g_3(z;\alpha)$. From Lemma~\ref{lemma:proj_properties3}, for all pairs $(z^*_i,z^*_j)$ we can write
    \[
    \norm{(z^*_i + \xi \bar d_i) - (z^*_j + \xi \bar d_j)}^2 \le \norm{z^*_i - z^*_j}^2, \quad \forall \xi \in (0,1],
    \]
    from which we get
    \[
    1 - e^{-\alpha\norm{(z^*_i + \xi \bar d_i) - (z^*_j + \xi \bar d_j)}^2} \le 1 - e^{-\alpha\norm{z^*_i - z^*_j}^2}, \quad \forall \xi \in (0,1],
    \]
    and then
    \begin{equation}\label{proof:decr6}
    g_3(z^* + \xi \bar d;\alpha) \le g_3(z^*;\alpha), \quad \forall \xi \in (0,1].
    \end{equation}
\end{itemize}
From~\eqref{proof:decr4}, \eqref{proof:decr5} and~\eqref{proof:decr6}, we conclude that~\eqref{proof:desc2} holds with $\bar \xi = 1$. This contradicts the fact the $z^*$ is a local optimal solution of~\eqref{zero-norm_model_approx}.
\end{proof}

In the previous subsection, we pointed out that for large values of the \mbox{parameter $\alpha$}, problem~\eqref{zero-norm_model_approx}
is a good approximation of~\eqref{zero-norm_model}.
The next theorem establishes the convergence of the global optimal solutions of problem~\eqref{zero-norm_model_approx} towards global optimal solutions
of problem~\eqref{zero-norm_model} for $\alpha \to +\infty$.

\begin{theorem}\label{th:global_convergence}
Given a finite set of vectors $X = \{x_1,\dots,x_m\} \subset \R^n$, $\lambda \ge 0$, \mbox{$w_{ij}\ge0$}, \mbox{$j=2,\dots,m,$} $\,i=1,\dots,j-1$, let $\bigl\{\alpha^t\bigr\}$ be a sequence of positive scalars such that $\alpha^{t+1} > \alpha^t$ and $\displaystyle\lim_{t\to\infty}\alpha^t = +\infty$. For any given parameter $\alpha \in \R^+$, let $g(z;\alpha)$ be the objective function of~\eqref{zero-norm_model_approx}, and $z(\alpha) = \begin{bmatrix}z_1(\alpha)^T & \dots & z_m(\alpha)^T\end{bmatrix}^T$ be a global optimal solution of~\eqref{zero-norm_model_approx}. Then,
\begin{enumerate}[(i)]
\item the sequence $\bigl\{g\bigl(z\bigl(\alpha^t\bigr);\alpha^t\bigr)\bigr\}$ converges,
\item the sequence $\bigl\{z\bigl(\alpha^t\bigr)\bigr\}$ attains limit points,
\item every limit point of $\bigl\{z\bigl(\alpha^t\bigr)\bigr\}$ is a global optimal solution of~\eqref{zero-norm_model}.
\end{enumerate}
\end{theorem}

\begin{proof}
Let $\phi(z)$ be the objective function of problem~\eqref{zero-norm_model}. Moreover, we indicate with \mbox{$z^* = \begin{bmatrix}(z^*_1)^T & \dots & (z^*_m)^T\end{bmatrix}^T$} a global optimal solution of problem~\eqref{zero-norm_model}.

From Proposition~\ref{prop:local_opt_sol_zero-norm_model_approx}, it follows that the sequence $\{z(\alpha^t)\}$ remains in a compact set,
thus it attains limit points, which proves (ii).

Now we show that, for all $t = 1,2,\dots$, the following relations hold:
\begin{align}
\label{major}&g\bigl(z;\alpha^t\bigr) \le g\bigl(z;\alpha^{t+1}\bigr) \le \phi(z), \quad \forall z \in \R^{mn}, \\
\label{lim_above}&g\bigl(z\bigl(\alpha^t\bigr);\alpha^t\bigr) \le g\bigl(z^*;\alpha^t\bigr) \le \phi(z^*), \\
\label{non_incr}&g\bigl(z\bigl(\alpha^t\bigr);\alpha^t\bigr) \le g\bigl(z\bigl(\alpha^{t+1}\bigr);\alpha^{t+1}\bigr).
\end{align}
Relation~\eqref{major} follows from the fact that $\alpha^{t+1} > \alpha^t > 0$. In fact, for every index pair $(i,j)$, we have
\[
1 - e^{-\alpha^t\norm{z_i-z_j}^2} \le 1 - e^{-\alpha^{t+1}\norm{z_i-z_j}^2} \le s\bigl(\norm{z_i-z_j}\bigr).
\]
The first inequality of~\eqref{lim_above} follows from the fact that $z\bigl(\alpha^t\bigr)$ minimizes $g\bigl(z;\alpha^t\bigr)$ with respect to $z$. The second inequality of~\eqref{lim_above} follows from~\eqref{major}. To prove~\eqref{non_incr}, assume by contradiction that it does not hold. Then there exists an index $t$ such that $g\bigl(z\bigl(\alpha^{t+1}\bigr);\alpha^{t+1}\bigr) < g\bigl(z\bigl(\alpha^t\bigr);\alpha^t\bigr)$. Using~\eqref{major}, we can write
\[
g\bigl(z\bigl(\alpha^{t+1}\bigr);\alpha^t\bigr) \le g\bigl(z\bigl(\alpha^{t+1}\bigr);\alpha^{t+1}\bigr) < g\bigl(z\bigl(\alpha^t\bigr);\alpha^t\bigr),
\]
which contradicts the fact that $z\bigl(\alpha^t\bigr)$ minimizes $g\bigl(z;\alpha^t\bigr)$ with respect to $z$. Then, \eqref{non_incr} must hold.

From~\eqref{lim_above} and~\eqref{non_incr}, it follows that the sequence $\bigl\{g\bigl(z\bigl(\alpha^t\bigr);\alpha^t\bigr)\bigr\}$ is monotonically non-decreasing and bounded from above. Thus it converges, proving (i).

Now, let $\bar{z}$ be a limit point of $\big\{z\bigl(\alpha^t\bigr)\bigr\}$, that is, there exists a subsequence $\big\{z\bigl(\alpha^t\bigr)\bigr\}_{\mathcal T}$ such that
\begin{equation}\label{acc_point}
\lim_{t\to\infty, \, t\in\mathcal{T}} z\bigl(\alpha^t\bigr) = \bar z.
\end{equation}
To prove (iii), we assume by contradiction that $\bar{z}$ is not a global optimal solution of~\eqref{zero-norm_model}. Then, there exists $\epsilon>0$ such that
\begin{equation}\label{opt_contr}
\phi(z^*) \le \phi(\bar z) - \epsilon.
\end{equation}
Since $\displaystyle \lim_{t\to\infty} g\bigl(z;\alpha^t\bigr) = \phi(z)$ for all $z \in \R^{mn}$, there exists an index $\bar{t}$ such that
\begin{equation}\label{major2}
\abs*{\phi(\bar z) - g\bigl(\bar z;\alpha^t\bigr)} < \epsilon, \qquad \forall t \ge \bar t.
\end{equation}
Using~\eqref{opt_contr} and~\eqref{major2}, we have that
$
g\bigl(\bar z;\alpha^t\bigr) > \phi(\bar z) - \epsilon \ge \phi(z^*),
$
for all $t \ge \bar t$.
Since $g\bigl(z;\alpha^t\bigr)$ is continuous with respect to $z$, there exists $\bar \rho > 0$ such that
\begin{equation}\label{major3}
g\Bigl(z;\alpha^{\bar t}\Bigr) > \phi(z^*), \qquad \forall z \in \mathcal{B}(\bar z,\bar \rho).
\end{equation}
From~\eqref{major}, \eqref{lim_above} and~\eqref{major3}, we can write
\begin{equation}\label{major4}
g\bigl(z;\alpha^t\bigr) \ge g\Bigl(z;\alpha^{\bar t}\Bigr) > \phi(z^*) \ge g\bigl(z^*;\alpha^t\bigr), \qquad \forall z \in \mathcal{B}(\bar{z},\bar{\rho}), \quad \forall t \ge \bar t.
\end{equation}
From~\eqref{acc_point}, there exists an index $\hat{t} \ge \bar t$ such that
\begin{equation}\label{neigh}
z\bigl(\alpha^t\bigr) \in \mathcal{B}(\bar z,\bar \rho), \qquad \forall t \ge \hat t \ge \bar t, \quad t \in \mathcal{T}.
\end{equation}
Finally, from~\eqref{major4} and~\eqref{neigh} we get
\[
g\bigl(z\bigr(\alpha^t\bigr);\alpha^t\bigr) > g\bigl(z^*;\alpha^t\bigr), \qquad \forall t \ge \hat{t}, \quad t \in \mathcal{T},
\]
which contradicts the fact that $z\bigl(\alpha^t\bigr)$ minimizes $g\bigl(z;\alpha^t\bigr)$ with respect to $z$ for sufficiently \mbox{large $t$}. This proves (iii).
\end{proof}

\section{The Data Filtering Method}\label{sec:filt_method}
Assuming that a fixed number of clusters is required,
in this section we present a data filtering strategy that combines model~\eqref{zero-norm_model_approx} with a technique to select a suitable value of the penalty parameter $\lambda$.

In particular, let $X = \{x_1,\dots,x_m\} \subset \R^n$ be a finite set of vectors and assume that $X$ must be partitioned into $k$ clusters, with $k$ fixed.
Let $\mathcal A$ be a generic clustering algorithm.
Our goal is to filter data to improve the performances of $\mathcal A$.

As discussed above, for any $\lambda$, an approximate solution $z^*_1,\dots,z^*_m$ of~\eqref{zero-norm_model} can be computed by solving~\eqref{zero-norm_model_approx} with suitable values of $\alpha$.
We observe that, independently of the obtained number of clusters, the points $z^*_1,\dots,z^*_m$ provide important information, because they are grouped on the basis of local densities of the samples $x_1,\dots,x_m$.
Therefore, each $z^*_i$ is representative of the behavior of $X$ in the neighborhood of $x_i$.
Consequently, after solving~\eqref{zero-norm_model_approx}, a partition $P$ of $X$ can be computed by applying $\mathcal A$ to the points $z^*_1,\dots,z^*_m$, instead of $x_1,\dots,x_m$.
Furthermore, some centroids should coincide, and then the geometry of $z^*_1,\dots,z^*_m$ is expected to be more regular than that of $x_1,\dots,x_m$. This can make the vectors $z^*_1,\dots,z^*_m$ easier to be clustered than $x_1,\dots,x_m$.

In other words, the (approximate) $\ell_0$-norm penalty model can be seen as a filtering method that maps each sample $x_i$ by a vector $z^*_i$
that is representative of the local density of $X$ in the neighborhood of $x_i$.

Naturally, the solutions of problem~\eqref{zero-norm_model_approx} are sensitive to the value of $\lambda$, that is,
different filters can be obtained by varying that parameter.
So, a strategy to choose a proper value of $\lambda$ must be introduced.

To this aim, we also have to take into account that the most suitable way to filter data can depend on the algorithm we apply later.
This is why we introduce a criterion to evaluate the partitions produced by $\mathcal A$ after filtering data
with a certain $\lambda$. The idea is to try different values of $\lambda$ and finally choose the best one in terms of our criterion,
as usually done in cross validation.
The whole filtering method is then summarized in \textbf{Algorithm~1}:

\smallskip
\begin{enumerate}[1.]
\setlength\itemsep{0.1em}
\item Given $\{x_1,\dots,x_m\}$, $1\le k \le m$, $\{\lambda_1,\dots,\lambda_N\} \in \R^+_0$, and an algorithm~$\mathcal A$
\item For $t = 1,\dots,N$
\item\label{alg2:zero_norm_step} \hspace*{0.5truecm} Set $\lambda = \lambda_t$ and compute $\{z^*_1,\dots,z^*_m\}$ by solving problem~\eqref{zero-norm_model_approx}
\item\label{alg2:clustering_step} \hspace*{0.5truecm} Compute a partition $\tilde{P}^t = \{\tilde{C}_1,\dots,\tilde{C}_k\}$ of $\{z^*_1,\dots,z^*_m\}$ by algorithm~$\mathcal A$
\item \hspace*{0.5truecm} Let $P^t = \{C_1,\dots,C_k\}$ such that $x_u \in C_i$ if $z^*_u \in \tilde{C}_i$, \mbox{$u=1,\dots,m, \, i=1,\dots,k$}
\item\label{alg2:eval_step} \hspace*{0.5truecm} Evaluate $P^t$ by assigning it a value $c(P^t)$
\item End for
\item Select $P^*$ as the best among $\{P^1,\dots,P^N\}$ in terms of $c(P^t)$
\end{enumerate}
\smallskip

We remark that the above strategy allows adapting the filtering to the specific clustering algorithm~$\mathcal A$.
Namely, different filters can be obtained for the same data set, according to the clustering algorithm to apply later.

We conclude this section by explaining how we compute $c(P^t)$ at Step~\ref{alg2:eval_step}.
Although the most proper way to evaluate a partition can strongly depend on the features of the specific data set (not known a priori),
the criterion we propose comes from a natural interpretation of clusters as subsets of similar points, where similarity is measured by the distance.
Basically, we encourage partitions with small distances within clusters and large distances between clusters.

More formally, given a partition $P^t = \{C_1,\dots,C_k\}$, where $C_1,\dots,C_k$ are disjoint subsets of $X$, we compute $c(P^t)$ at Step~\ref{alg2:eval_step} as
\[
c(P^t) = \frac{1}{d_{b}} \sum_{i=1}^k \frac{\displaystyle d_w(i)}{\displaystyle n_p(i)},
\]
where $d_w(i)$ is the sum of the distances within cluster $C_i$,
$n_p(i)$ is the number of pairs of points belonging to cluster $C_i$
and $d_b$ is the sum of the distances between all the pairs of points belonging to different clusters.

In order to operate in high-dimensional spaces, we also use kernel functions to compute distances between points (see~\cite{scholkopf:2002} for definition and properties of kernel functions), so that the distance between two vectors $x_u, x_v \in R^n$ can be computed as
\[
K(x_u,x_u) - 2K(x_u,x_v) + K(x_v,x_v),
\]
where $K(\cdot,\cdot)$ is the chosen kernel function.
In particular, in our simulations we used a Gaussian kernel.
Given $x_u,x_v \in \R^n$, the Gaussian kernel is defined as
\begin{equation}\label{gauss_kernel}
K(x_u,x_v) = e^{-\gamma \norm{x_u-x_v}^2}, \quad x_u,x_v \in \R^n.
\end{equation}
In our experiments, we set $\gamma = 0.1$.

\section{Numerical Experience}\label{sec:numerical_results}
In this section, we report our numerical experience.
In Subsection~\ref{subsec:exp}, we describe how we set up the experiments.
In Subsection~\ref{subsec:solv_prob}, we show how we solved problem~\eqref{zero-norm_model_approx}.
Finally, in Subsection~\ref{subsec:res}, we report and discuss the numerical results.

\subsection{Experimental Set-Up}\label{subsec:exp}
We compare the performances of three well known clustering methods when they are applied to the original data
and when they are applied to the data filtered by Algorithm~$1$.
The considered methods are the following:
\begin{itemize}
\item Single-Linkage (SL) method, which is a hierarchical clustering algorithm that iteratively merges the two clusters containing the closest pair of points
    (see~\cite{gan:2007} for further details);
\item Expectation-Maximization for Gaussian mixture Models (EMGM), which tries to estimate the parameters of the probability density functions generating the samples (see~\cite{dempster:1977,mclachlan:2004} for further details);
\item Kernel K-Means (KKM), which is an extension of the standard k-means method exploiting kernel functions to compute distances
    (see \cite{scholkopf:1998,girolami:2002,dhillon:2004} for further details).
\end{itemize}

Since KKM and EMGM aim to solve non-convex optimization problems, they were executed $1000$ times, choosing randomly the starting parameters, and finally taking the solution providing the best objective value. In particular, to run KKM, we used a Gaussian kernel, defined as in~\eqref{gauss_kernel},
with $\gamma = 0.1$.

In addition, we also tried to filter data by using different techniques. First, to show the effect of the $\ell_0$-norm penalty, we tested a different regularization.
In particular, we considered the squared \mbox{$\ell_2$-norm} regularization (also known as ridge regularization), which typically does not induce sparsity.
In this case, we applied Algorithm~$1$ replacing problem~\eqref{zero-norm_model_approx} at Step~$3$ with the following problem:
\begin{equation}\label{sq_l2_norm_reg_prob}
\min_{z \in \R^{mn}} \sum_{i=1}^m\norm{x_i-z_i}^2 + \lambda\sum_{j=2}^m\sum_{i=1}^{j-1} w_{ij} \norm{z_i-z_j}^2.
\end{equation}
Both for the filter obtained with the $\ell_0$-norm regularization and the one obtained with the squared $\ell_2$-norm regularization,
$150$ increasing values of $\lambda$ were used, chosen such that $\lambda_1=0$ and $\lambda_{150}$ provides a solution of problem~\eqref{zero-norm_model_approx}
(respectively, problem~\eqref{sq_l2_norm_reg_prob}) that collapses to a single centroid.
For both problem~\eqref{zero-norm_model_approx} and problem~\eqref{sq_l2_norm_reg_prob}, the weight parameters $w_{ij}$ were set as
\[
w_{ij} = e^{-0.1 \norm{x_i-x_j}^2}, \quad j=2,\ldots,m, \quad i=1,\ldots,j-1.
\]

A further filtering technique we tried in our experiments is running k-means~\cite{macqueen:1967}
with a predetermined number of clusters $\bar k$, in order to represent the original data by the centroids.
Namely, after applying k-means (which was repeated $1000$ times), the clustering algorithms SL, EMGM and KKM were applied to the resulting $\bar k$ centroids.
We tested this filter with different values of $\bar k$. In particular, we tried $\bar k = 5\,k$, $\bar k = 10\,k$ and $\bar k = 20\,k$,
where $k$ is the true number of clusters of a given data set.
The rationale behind this choice is to have a number of representative centroids larger than the number of features, for every data set.

The experiments were conducted on some synthetic and real data sets, covering different scenarios\footnote{All data were scaled in $[-1,1]$.}:
\begin{description}
\item[Case (i):] two spherical clusters in two dimensions, with equal volumes and the same cardinality.
   The first cluster has $50$ points, generated from a bivariate Normal distribution with mean vector $\begin{pmatrix} 0 & 0 \end{pmatrix}^T$
   and covariance matrix $0.33^2I$. The second cluster has $50$ points, drawn from a bivariate Normal distribution
   $N\Bigl(\begin{pmatrix} 1 & 1 \end{pmatrix}^T, 0.33^2I\Bigr)$.
\item[Case (ii):] two elongated clusters in two dimensions, with different cardinalities.
   The first cluster has $500$ points, generated from a bivariate Normal distribution with mean vector $\begin{pmatrix} 0 & 5 \end{pmatrix}^T$
   and covariance matrix $\begin{pmatrix} 0.05 & 0 \\ 0 & 5 \end{pmatrix}$. The second cluster has $50$ points, drawn from a bivariate Normal distribution
   $N\biggl(\begin{pmatrix} 2.5 \\ 0 \end{pmatrix}, \begin{pmatrix} 0.3 & 0 \\ 0 & 0.05 \end{pmatrix}\biggr)$.
\item[Case (iii):] two spherical clusters in two dimensions, with different volumes and cardinalities.
   The first cluster has $500$ points, generated from a bivariate Normal distribution $N\Bigl(\begin{pmatrix} 0 & 0 \end{pmatrix}^T, 4I\Bigr)$.
   The second cluster has $50$ points, generated from a bivariate Normal distribution $N\Bigl(\begin{pmatrix} 7 & 0 \end{pmatrix}^T, 0.5I\Bigr)$.
\item[Case (iv):] four clusters in three dimensions. The centers $\mu_1, \mu_2, \mu_3, \mu_4 \in \R^3$ were drawn from a multivariate distribution
    $N\Bigl(\begin{pmatrix} 0 & 0 & 0 \end{pmatrix}^T, 5I\Bigr)$. When generating the centers, if two of them had an Euclidean distance smaller than $1$,
    the simulation was aborted and then started again. After fixing the centers, the number of elements for each cluster was randomly chosen in the range $[10,100]$.
    Finally, for each cluster $i$, the points were generated from a multivariate distribution $N(\mu_i,I)$.
    This is similar to case IV in \cite{pan:2013}, and scenario~(c) in \cite{tibshirani:2001}, but here clusters are more imbalanced.
\item[Case (v):] the Ecoli data set from the UCI repository \cite{lichman:2013}.
    There are $336$ samples characterized by $7$ features and divided into 8 clusters, which contain $143$, $77$, $52$, $35$, $20$, $5$, $2$ and $2$ elements, respectively.
\item[Case (vi):] the Fisher's Iris data from the UCI repository \cite{lichman:2013}.
    The points are in four dimensions and divided into $3$ clusters of $50$ elements each.
    The second and the third cluster are partially overlapped, whereas the first cluster is linearly separable from the other two.
\item[Case (vii):] the wine data set from the UCI repository \cite{lichman:2013}, with $178$ samples of $3$ kinds of wine.
    The clusters contain $59$, $71$ and $48$ elements, respectively, and each sample is characterized by $13$ features.
\item[Case (viii):] the Wisconsin breast cancer data set from the UCI repository \cite{lichman:2013,mangasarian:1990}.
    There are $683$ samples of $9$ features each\footnote{Originally, there were $699$ samples, but $16$ of them had missing values and were removed.},
    divided in two groups: $444$ benign and $239$ malignant.
\end{description}

The partitions are finally evaluated by the Adjusted Rand Index (ARI)~\cite{hubert:1985}, which takes $1$ as maximum value (ARI can also assume negative values).

\subsection{Solving the Approximating Problem}\label{subsec:solv_prob}
Taking into account Theorem~\ref{th:global_convergence}, solving~\eqref{zero-norm_model_approx} with large values of $\alpha$
can be a practical solution to get good approximations of the optimal solutions of~\eqref{zero-norm_model}.
Theorem~\ref{th:global_convergence} would also require to compute a global solution of the approximating problem,
so, a global algorithm should be used, to be in line with the theory.
But global algorithms are in general computationally expensive, especially when dealing with large-scale problems, as in our case.
Moreover, since model~\eqref{zero-norm_model_approx} is employed to filter data, then (i) solving the problem should not be too expensive, and
(ii) it could be sufficient to compute ``good'' solutions of problem~\eqref{zero-norm_model_approx}, even if not global optima.
Thus, it can be reasonable to employ a local algorithm that, on the one hand, can provide non-global minimizers,
but, on the other hand, is cheaper than a global method.

After all, many clustering models are formulated as non-convex problems, and several algorithms that are widely used in practice
are based on local strategies (e.g., the aforementioned KKM and EMGM).
Anyway, defining efficient global methods to solve~\eqref{zero-norm_model_approx} can be a challenging task for future research.

For the above reasons, we solved problem~\eqref{zero-norm_model_approx} by employing a non-monotone version of the truncated-Newton method
which exploits negative curvature directions (so, it is well suited for non-convex problems), proposed in~\cite{fasano:2009}.

Finally, another computational issue is that problem~\eqref{zero-norm_model} becomes ill-conditioned when $\alpha$ and $\lambda$ get large.
Then, we employed a warm-start strategy, gradually increasing $\alpha$ up to a prefixed value
(this approach was also proposed in~\cite{bradley:1998}, but not attempted in practice).
In particular, starting with $\alpha^t = 1$, $t=1$, we employed the following updating rule:
$\alpha^{t+1} = \min \bigl\{10^3, \bigl(1+e^{-0.07t}\bigr)\alpha^t\bigr\}$, $t = t+1$, stopping the algorithm when $\alpha^t$ reaches $10^3$.
For every $\alpha^t$, we solved the problem with a growing precision, terminating the minimization when the sup-norm of the gradient of the objective function
was less than or equal to $\epsilon^t = \max\{10^{-5},10^{-2}/\alpha^t\}$.

\subsection{Results}\label{subsec:res}
The final results are summarized in Table~\ref{tab:res}.
The filter based on the $\ell_0$-norm regularization and the one based on the squared $\ell_2$-norm regularization are indicated as
\mbox{\textit{$\ell_0$ filter}} and \textit{ridge filter}, respectively.
The filter obtained by k-means is denoting with \textit{KM filter} and the number of clusters used is given within brackets.

For the ridge filter, problem~\eqref{sq_l2_norm_reg_prob} was solved by employing the truncated-Newton method reported in~\cite{fasano:2009},
terminating the algorithm when the sup-norm of the gradient of the objective function was less than or equal to $10^{-5}$.

All computations were run on an Intel(R) Core(TM) i7-3770 CPU 3.40 GHz and the codes were implemented in Fortran~90.

{\setlength{\tabcolsep}{0.25em}
\begin{table}
\centering
\caption{Comparison between the values of the Adjusted Rand Index obtained by applying Single Linkage~(SL),
Expectation-Maximization for Gaussian mixture Models (EMGM) and  Kernel K-Means (KKM) to the original data
and to the filtered data.
Three different filters are considered: the $\ell_0$ filter is the one reported in Algorithm~$1$,
the ridge filter differs from the previous one in that problem~\eqref{zero-norm_model_approx} is replaced
with problem~\eqref{sq_l2_norm_reg_prob} at Step~$3$ of Algorithm~$1$,
and KM filter is obtained by applying k-means with a prefixed number of clusters $\bar k$
(which is indicated within brackets, where $k$ is the true number of clusters).}
{\begin{tabular}{c c c c c c c c c}
\toprule
\multirow{2}*{{\bf Method}} & \multicolumn{8}{c}{{\bf Dataset}}                                             \\
                            & (i)     & (ii)    & (iii)   & (iv)    & (v)     & (vi)    & (vii)   & (viii)  \\
\midrule
SL                          &  0.0000 &  1.0000 & -0.0032 &  0.0057 &  0.0399 &  0.5584 & -0.0038 &  0.0025 \\
$\ell_0$ filter + SL        &  1.0000 &  1.0000 &  1.0000 &  0.4022 &  0.4155 &  0.5657 & -0.0068 &  0.8685 \\
ridge  filter + SL          &  0.0008 &  1.0000 & -0.0032 &  0.0057 &  0.0399 &  0.5584 & -0.0038 &  0.0025 \\
KM filter (5k) + SL         &  1.0000 &  1.0000 &  0.9869 &  0.2429 &  0.0520 &  0.5621 & -0.0107 &  0.0670 \\
KM filter (10k) + SL        &  0.0000 &  1.0000 &  0.9869 &  0.2157 &  0.0482 &  0.5638 & -0.0003 &  0.0073 \\
KM filter (20k) + SL        &  0.0000 &  1.0000 & -0.0063 &  0.0020 &  0.0399 &  0.5584 & -0.0068 &  0.0101 \\
\midrule
EMGM                        &  0.9600 &  1.0000 &  1.0000 &  0.5930 &  0.5843 &  0.4414 &  0.4778 &  0.5547 \\
$\ell_0$ filter + EMGM      &  1.0000 &  1.0000 &  1.0000 &  0.5697 &  0.6752 &  0.5657 &  0.7032 &  0.8798 \\
ridge filter + EMGM         &  0.9600 &  1.0000 &  1.0000 &  0.5841 &  0.5768 &  0.9039 &  0.8154 &  0.8798 \\
KM filter (5k) + EMGM       &  0.0173 &  1.0000 &  0.6496 &  0.6776 &  0.7594 &  0.5676 &  0.6585 &  0.5301 \\
KM filter (10k) + EMGM      &  1.0000 &  1.0000 &  0.6200 &  0.4577 &  0.2890 &  0.4531 &  0.5303 &  0.7929 \\
KM filter (20k) + EMGM      &  0.8448 &  1.0000 &  0.9869 &  0.6106 &  0.2831 &  0.5399 &  0.3909 &  0.0785 \\
\midrule
KKM                         &  1.0000 &  0.9741 &  0.3977 &  0.4894 &  0.4538 &  0.7163 &  0.8992 &  0.8686 \\
$\ell_0$ filter + KKM       &  1.0000 &  1.0000 &  1.0000 &  0.6865 &  0.6977 &  0.7445 &  0.8820 &  0.8742 \\
ridge filter + KKM          &  1.0000 &  0.9741 &  0.3977 &  0.4894 &  0.4730 &  0.7302 &  0.8992 &  0.8686 \\
KM filter (5k) + KKM        &  1.0000 &  1.0000 &  0.3888 &  0.4803 &  0.4551 &  0.6537 &  0.7857 &  0.8031 \\
KM filter (10k) + KKM       &  1.0000 &  0.9491 &  0.0209 &  0.4407 &  0.5527 &  0.7060 &  0.8369 &  0.7870 \\
KM filter (20k) + KKM       &  1.0000 &  0.9615 &  0.1622 &  0.4837 &  0.4929 &  0.7455 &  0.8686 &  0.8300 \\
\bottomrule
\end{tabular}}
\label{tab:res}
\end{table}
}

First, let us discuss the results achieved by the $\ell_0$ filter.
Overall, the performances of the considered clustering methods improve by using this data filtering process.

In particular, in six data sets, the results obtained by SL are unsatisfactory by applying the algorithm to the original data,
whereas performances remarkably increase when data are filtered.
Only for the wine data set (case (vii)), the filtering does not lead to better results.

As regards EMGM, the data filtering strategy allows to improve the performances on all the real data sets (case (v)--(viii)).
Only for case (iv), better partitions are obtained by applying the algorithm to the original data.

Also for KKM, the best partitions are those computed on the filtered data, except for case (vii)
(and excluding case (i), where the right clusters are recognized also without filtering).
A significant result is obtained for case (iii), where the clusters to detect have remarkably different volumes.
This is known to be a hard case for centroid-based methods, but that issue has been overcome by the filtering strategy.

For what concerns the computational time, we plot in Figure~\ref{fig:cpu_time} the CPU time (in seconds) needed
to solve problem~\eqref{zero-norm_model_approx} versus the value of the penalty parameter $\lambda$.

\begin{figure}[h]
\centering
\includegraphics[trim = 1.2cm 0.2cm 1.1cm 0cm, clip]{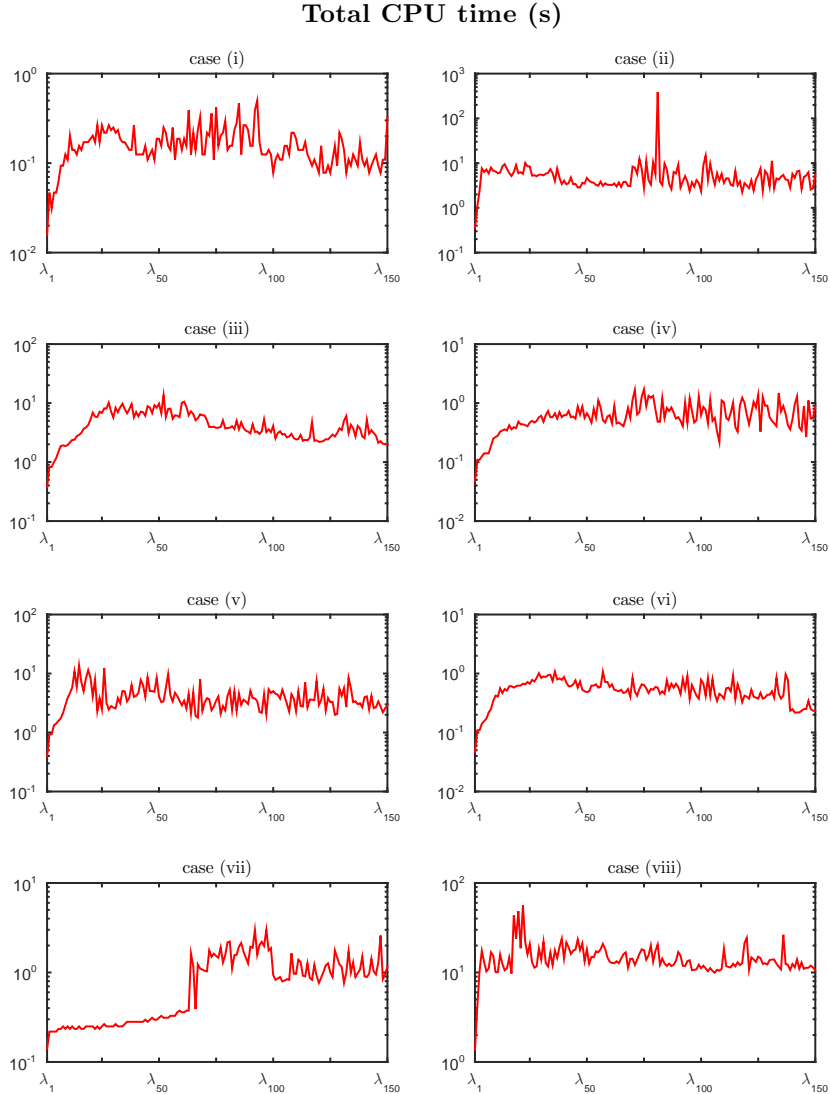}
\caption{CPU time in seconds needed to solve problem~\eqref{zero-norm_model_approx} versus the value of $\lambda$.
The $y$ axis is in logarithmic scale.}
\label{fig:cpu_time}
\end{figure}

We observe that each minimization required less than $3$ seconds for case~(i), (iv), (vi) and~(vii).
For case~(ii), (iii) and~(v), every minimizations took less than $15$ seconds, except for
a single value of $\lambda$ in case~(ii), which required $385$ seconds.

As regards the largest data set considered in the experiments, i.e., case~(viii),
the minimizations took between $40$ and $60$ seconds for three values of $\lambda$.
For the remaining values of $\lambda$, every minimization required less than $30$ seconds.
Overall, the average time needed to solve problem~\eqref{zero-norm_model_approx} is about $15$ seconds.

Recalling that we solved~\eqref{zero-norm_model_approx} with a warm-start strategy
(by employing increasing values of $\alpha$ and solving the problem with a growing precision),
it is also interesting to analyze the amount of time needed in the minimization procedure for every value
of $\alpha$. We report these times (in seconds) in Figure~\ref{fig:cpu_time_alpha}.
In particular, for every considered~$\alpha$, in Figure~\ref{fig:cpu_time_alpha} is depicted the average time
over the $150$ considered values of the penalty parameter $\lambda$, needed to solve~\eqref{zero-norm_model_approx}
with the related precision.
In almost all data sets, the computational time increases when $\alpha$ becomes large, as expected.
Only for case~(viii), we have that small values of $\alpha$ required more time.
However, the computational time needed to solve~\eqref{zero-norm_model_approx} remains, on average, below $5$ seconds for
every considered $\alpha$.

\begin{figure}[h!]
\centering
\includegraphics[trim = 1.2cm 0.15cm 1.1cm 0cm, clip]{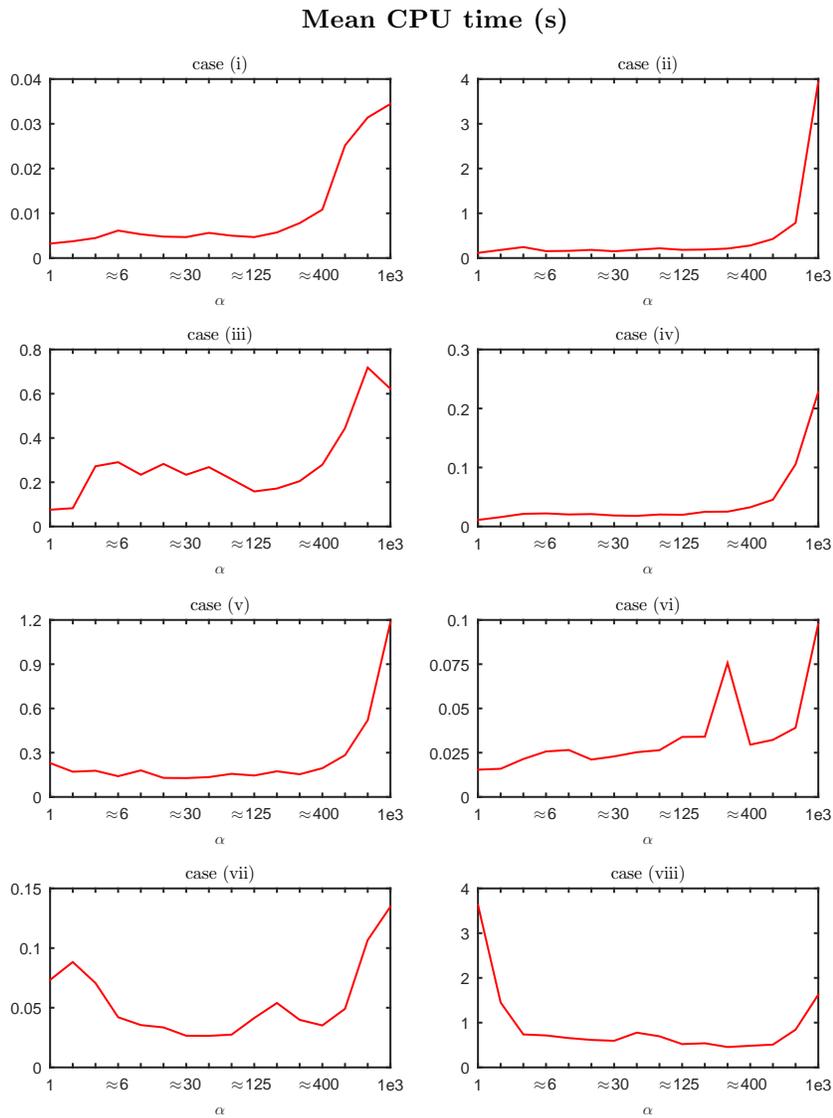}
\caption{CPU time in seconds needed to solve problem~\eqref{zero-norm_model_approx} versus the value of $\alpha$.
In each plot, the computational time is averaged over the $150$ considered values of the penalty parameter $\lambda$.}
\label{fig:cpu_time_alpha}
\end{figure}

Now, we discuss the results obtained by applying the ridge filter.
On the one hand, a clear advantage of using this filter is the low computational time needed to solve the optimization problem.
In particular, less than $0.25$ seconds were required to solve~\eqref{sq_l2_norm_reg_prob}, for every data set and for every value of $\lambda$.
Furthermore, problem~\eqref{sq_l2_norm_reg_prob} is smooth and convex, and then a global optimal solution can be computed by a local algorithm.

On the other hand, the numerical results seem worse than those achieved by the $\ell_0$-norm regularization.

In particular, the ridge filter has essentially no effect on SL.
Similarly, it does not provide relevant effects on KKM either.
Looking at the results more in detail, we also observe that this filter
is not able to improve the performance of KKM for case~(iii), which is a known problematic data set for centroid-based
methods, as discussed above.

For what concerns EMGM, the partitions obtained by employing the ridge filter are better than those computed
on the raw data for case~(vi), (vii) and~(viii).
In particular, very good results are achieved on the Iris data set.
In comparison with the $\ell_0$-norm regularization, we observe that the ridge regularization provides better results
for case~(iv) (even though they are still worse than those obtained on the raw data), case~(vi) and case~(vii),
whereas the $\ell_0$-norm regularization provides better results for case~(i), even if slightly, and case~(v).

Now, let us discuss the results achieved by KM filter.
In terms of wins, for SL and EMGM the best partitions are those obtained by running k-means with a number of clusters $\bar k$ equal to $5 \, k$,
while $\bar k = 20 \, k$ seems the best choice for KKM.
Overall, KM filter seems to perform worse than the $\ell_0$ filter, but it is much faster
(each run of the k-means algorithm took less than $0.1$ seconds, for every considered data set).

Summarizing, the $\ell_0$-norm regularization based filter seems able to benefit different clustering algorithms and
it seems more flexible than the ridge regularization based filter. Moreover, it produced better results than the k-means based filter
(for the considered choices of number of clusters $\bar k$).
From a computational point of view, both the ridge filter and the KM filter turn out to be more efficient;
however, also the computational time needed by the $\ell_0$ filter remains, on average,
below an acceptable threshold, for all the considered data sets.

Finally, let us spend some words on the applicability of the proposed approach for large data sets.
In our experiments, we were able to solve problem~\eqref{zero-norm_model_approx} efficiently by employing a Newton-type method.
We noted that this choice is effective when the problem dimensions
(i.e., the product of the number of samples and the number of features) do not exceed $10^4-10^5$.
To cope with larger problems, we think that the optimization procedure should be properly adjusted, for example
by using a block decomposition algorithm that exploits the particular structure of the objective function.
Additionally, the warm-start strategy could be stopped earlier, (i.e., smaller values of the parameter $\alpha$ could be employed),
even if this trades off with the accuracy of the $\ell_0$-norm approximation.


\section{Conclusions}\label{sec:conclusions}
We have presented a data filtering method for cluster analysis, based on combining two strategies:
the first one is the minimization of a least squares function with a weighted $\ell_0$-norm penalty, approximated by smooth parametric functions;
the second one is choosing the penalty parameter by minimizing a suitable criterion that considers the distances within and between clusters.
Promising results have been obtained from numerical simulations, performed on synthetic and real data sets.

\bibliography{cristofari2016}

\end{document}